\documentclass[11pt]{amsart}
\usepackage[utf8]{inputenc}
\usepackage{xcolor}
\usepackage{bm}
\usepackage{cleveref}
\usepackage{graphicx}
\usepackage{paralist}
\usepackage{floatrow}

\title[Inflation of poorly conditioned zeros]{Inflation of poorly conditioned zeros of systems of analytic functions}

\author{
Michael Burr}
\address{Clemson University, 220 Parkway Drive, Clemson, SC 29634}
\email{burr2@clemson.edu}
\thanks{Burr was partially supported by NSF DMS \#1913119}
\author{ 
Anton Leykin 
}
\address{School of Mathematics, Georgia Tech, Atlanta, Georgia}
\email{leykin@math.gatech.edu}
\thanks{Leykin was partially supported by NSF DMS \#1719968 and \#2001267.}
\date{\today}

\makeatletter
\renewcommand*\env@matrix[1][\arraystretch]{%
  \edef\arraystretch{#1}%
  \hskip -\arraycolsep
  \let\@ifnextchar\new@ifnextchar
  \array{*\c@MaxMatrixCols c}}
\makeatother

\newcommand{\DEF}[1]{{\color{blue}#1}}

\newcommand{\Sphere}{{\mathbb S}}
\newcommand{\Ball}[1]{{\mathbb B}_{#1}}
\newcommand{\mixedBall}[2]{\Ball{#2}^{(#1)}}

\newcommand{\Span}[1]{\langle #1 \rangle}
\newcommand{\ideal}[1]{(#1)}
\newcommand{\ii}{{\mathbf i}}
\newcommand{\RR}{{\mathbb R}}
\newcommand{\CC}{{\mathbb C}}

\renewcommand{\Re}{\operatorname{Re}}
\renewcommand{\Im}{\operatorname{Im}}

\newtheorem{theorem}{Theorem}[section]
\newtheorem{lemma}[theorem]{Lemma}
\newtheorem{corollary}[theorem]{Corollary}

\newtheorem{example}[theorem]{Example}

\usepackage{framed}
\usepackage{enumitem}

\begin{document}

\begin{abstract}
Given a system of analytic functions and an approximate zero, we introduce inflation to transform this system into one with a regular quadratic zero.  This leads to a method for isolating a cluster of zeros of the given system.
\end{abstract}

\maketitle

Let $g : \CC^n\to \CC^n$ be a square system of analytic functions. There exist heuristic methods to obtain approximations to some or all of the isolated zeros of $g$. In the context of our work, \DEF{isolation} of a zero is the problem of constructing a neighborhood of the given approximation that contains exactly one zero. When the approximation is close to a singular zero (with multiplicity $m$) of a nearby system, this problem becomes the problem of isolating a \DEF{cluster} of $m$ zeros.  

For regular zeros, methods derived from alpha-theory \cite{ComplexityComputation,alphaCertified} or interval arithmetic \cite{Moore:Interval} provide effective ways to isolate zeros.  These methods, however, do not apply to singular zeros.  We introduce an approach, called \DEF{inflation}, that transforms a system with an approximate zero into a system with a regular quadratic zero --- a new concept which we define.  We use this new system to compute a region containing a cluster of zeros of the original system.

Most prior approaches for approximating poorly conditioned (i.e., singular or nearly singular) zeros, such as deflation \cite{LVZ}, aim to \emph{reduce} the multiplicity of a zero so that the resulting zeros can be found using standard methods.  One new idea of our approach is, instead, to \emph{increase} the multiplicity of the zero or cluster of zeros.  By increasing the multiplicity of the zero, we regularize the zero, albeit not in a traditional sense --- regular quadratic zeros are still singular, but better behaved. Our approach combines ideas from linear and nonlinear algebra, real and complex analysis,  as well as convex optimization.


We refer the reader to \cite{HaoJiangLiZhi:2020simple-multiple-zeros} and the references therein for a discussion of other approaches to the zero clustering problem and the related problem of certifying a singular zero.
To the best of our knowledge, prior to this work, the clustering problem has only been considered in the univariate case and the multivariate case where the Jacobian has corank one.



\section{Exact singular zeros}\label{sec:exactzeros}

We call $x$ a \DEF{regular zero of order $d$} of a system $g$ if, 
first, the degree $i$ part of the Taylor expansion of $g$ at $x$ is identically zero for $i=0,\dots,d-1$, and, 
second, the degree $d$ part does not vanish on the unit sphere $\Sphere \subset \CC^n$.
We call a regular zero of order 2 a \DEF{regular quadratic zero}.

Let $y\in\CC^n$ be a zero of $g$ and $\kappa=\dim\ker Dg(y)$. 
If $\kappa=0$, then the zero is a \DEF{regular zero}, or, in our new terminology, regular of order one.  The most common case of a \DEF{multiple} (nonregular and singular) zero is called a \DEF{simple double zero} in~\cite{dedieu2001simple} and corresponds to a singularity of Type $\Sigma^{1,0}$ using the terminology of singularity theory \cite{SingularityBook}.  We consider the more general case where $y$ is a singularity of Type $\Sigma^\kappa$.  We define the \DEF{multiplicity} of an isolated zero $y$ to be the number of zeros near $y$ after a generic perturbation of $f$ \cite{SingularityBook}.

We introduce inflation as a way to transform a general system with a singularity of Type $\Sigma^\kappa$ at $y$ into a system with a regular quadratic zero at $y$.  Next we apply Rouch\'e's theorem to the quadratic part of the transformed system in order to calculate the multiplicity of the zero $y$ of the original system.  Throughout this section, we assume that the singular zero $y$ is explicitly known, and, hence, the kernel $\ker Dg(y)$ can be explicitly computed.  This restriction is lifted in \Cref{sec:approximate}.

\subsection{Inflation}\label{sec:inflation}

We first simplify the problem by applying an affine change of coordinates $A:\CC^n\to\CC^n$ with $A(0)=y$ so that the origin is a zero of $g\circ A$ and 
$$
\ker D(g\circ A)(0) = \Span{e_1,\dots,e_\kappa}.
$$
We then obtain the \DEF{inflated system} $g\circ A\circ S_\kappa$ where
\begin{equation*}
S_\kappa(x_i) = \left\{\begin{array}{cc}
     x_i, &  i\in [\kappa]; \\[.05cm]
     x_i^2,& \text{otherwise}.
\end{array}\right.  
\end{equation*}
Since $g$ is generic, the quadratic part of $g\circ A\circ S_\kappa$ does not vanish on the unit sphere, so $g\circ A\circ S_\kappa$ has a regular quadratic zero at the origin.  We remark that for $\kappa<n$, the resulting zero has \DEF{higher} multiplicity than the original zero.  More precisely, from the point of view of singularity theory, we pass from Type $\Sigma^\kappa$ to Type $\Sigma^n$.  This behavior prompted us to name this method ``inflation'' in contrast to ``deflation", see, e.g., \cite{LVZ}.

\subsection{Isolating zeros}

Let $\Sphere_{\varepsilon} \subset \CC^n$ be the sphere of radius $\varepsilon$ and $\Ball{\varepsilon}$ be the ball of radius $\varepsilon$ centered at $0$ in the Euclidean norm. 

\begin{lemma}\label{lemma:cluster}
Let $f$ be an arbitrary square system and $Q$ a square system of quadratic forms.  Suppose there exists $c>0$ and $\varepsilon>0$ such that 
\begin{enumerate}
    \item $\|(f-Q)(\Sphere_{\varepsilon})\|\leq c\varepsilon^2$, and \label{condition1} 
    \item $\|Q(\Sphere_1)\|\geq c$,\label{condition2}
\end{enumerate}
then 
$f$ has $2^n$ zeros in $\Ball{\varepsilon}$.  
\end{lemma}
\begin{proof}
  Condition \ref{condition2} implies that $Q$ has no zeros on a sphere.  Therefore, by homogeneity, its only zero is the origin, which has multiplicity $2^n$.
  Since $\|f-Q\|\leq c\varepsilon^2 \leq \|Q\|$ on $\Sphere_{\varepsilon}$,
  the conclusion follows from multivariate version Rouch\'e's theorem, see~\cite[Theorem 2.12]{Berenstein:Rouche}.  
\end{proof}

Since $g\circ A\circ S_\kappa$ has a regular quadratic zero at the origin, we define $Q=Q(g)$ to be the \DEF{quadratic part} of $g\circ A\circ S_\kappa$.  For $\varepsilon$ a small positive number, the conditions of \Cref{lemma:cluster} hold for the pair $(g\circ A\circ S_\kappa,Q)$.  

Since the inflation map is a cover with  $2^{n-\kappa}$ sheets with ramification locus $\{x_1x_2\dots x_\kappa=0\}$, we conclude that $y$ is a zero of multiplicity $2^\kappa$ of $g$.


We illustrate our technique when the zero of $g$ is explicitly known:  

\begin{example}\label{sec:example:multiple}
Let $\alpha\not=0$ and $g$ be the following system with a zero at $(y_1,y_2)$: 
$$
g(x) =  [(x_1-y_1)^2, (x_2-y_2)-\alpha(x_1-y_1)^2((x_1-y_1)+3y_1)]^T.  
$$
We observe that $\ker Dg(y)=\langle e_1\rangle$ and $\kappa=1$. Applying the affine transformation $A(x)=x+y$ 
and inflation, we get
$$
g\circ A\circ S_1(x)=[x_1^2,x_2^2-\alpha x_1^2(x_1+3y_1)]^T.
$$
We observe that the quadratic part of $g\circ A\circ S_1$ is $$Q(g\circ A\circ S_1)=[x_1^2,x_2^2-3\alpha y_1 x_1^2]^T,$$ which does not vanish on the unit sphere $\Sphere_1$.  Therefore, $g\circ A\circ S_1$ is a regular quadratic zero of multiplicity four at the origin.  Near the origin, $Q(g\circ A\circ S_1)$ dominates $g\circ A\circ S_1-Q(g\circ A\circ S_1)$.  

Using Rouch\'e's theorem~\cite[Theorem 2.12]{Berenstein:Rouche}, we conclude that $g\circ A\circ S_1$ has four zeros (counted with multiplicity) in the ball $\Ball{\varepsilon}$ for some $\varepsilon>0$. Since $A\circ S_1$ is a double cover, we conclude that $g$ has a zero of multiplicity \emph{two} at $y$.
\end{example}

We hint at how the computation from \Cref{sec:example:multiple} can be used to isolate a cluster of zeros.  Let $f(x)=[x_1^2-\alpha^2,x_2-\alpha x_1^3]^T$, which has zeros at $\pm(\alpha,\alpha^4)$. Suppose that we approximate this cluster of zeros with $y$. We observe that $f-f(y)$ differs from $g$ in \Cref{sec:example:multiple} only in the linear terms in the Taylor expansion around $y$. Moreover, the limit of this difference tends to $0$ as $\alpha\to 0$.  This behavior should be contrasted with that of the remaining terms contributing to $Q(g\circ A\circ S_1)$, which tend to a quadratic form that is nonvanishing on a sphere.  Therefore, for $y$ and $\alpha$ sufficiently small, $f-f(y)$ is dominated by $Q$ on $\Sphere_\varepsilon$.  Thus, Rouch\'e's theorem identifies a cluster of zeros of $f$ near $y$.

In the remainder of this paper, we expand upon this approach.

\section{Poorly conditioned zeros}\label{sec:approximate}

In typical applications, for system $f$, we do not expect that a multiple zero is given or even that $f$ has any singular zeros.  Instead, we assume that we are given an approximation $y$ to a cluster of zeros so that $f$ is nearly singular at $y$.  In addition, let $\kappa$ be the dimension of the numerical kernel of $D\!f(y)$, i.e., the number of small singular values.  We construct a nearby system $g$ which has a zero at $y$ and whose Jacobian has corank $\kappa$.  Then, by applying the methods from \Cref{sec:exactzeros}, we isolate a cluster of zeros of $f$.

\subsection{Transformations of the system} \label{sec:nearby-singular-system}
Suppose that $V\subset\CC^n$ is a linear subspace with $\dim V = \kappa$.  Consider
\begin{equation}\label{eq:singular}
g(x):=f(x)-f(y)-\sum_{i=1}^\kappa D\!f(y)\pi_V(x-y)
\end{equation}
where $\pi_V$ is the orthogonal projection onto $V$.  We observe that $y$ is a zero of $g$ and that $\ker Dg(y)\supset V$.  When $D\!f(y)$ is nonsingular, this containment is an equality.  In our computations, $V$ plays the role of an approximate kernel of $D\!f(y)$. One practical choice for $V$ is the span of singular vectors corresponding to the $\kappa$ smallest singular values of $D\!f(y)$.

\subsection{Isolating a cluster of zeros}

Suppose that $g$ is constructed from $f$ as in \Cref{sec:nearby-singular-system} and that $A$ is constructed from $g$ as in \Cref{sec:inflation}.  We then compare the zeros of $f\circ A\circ S_\kappa$ to those of $Q(g\circ A\circ S_1)$ using \Cref{lemma:cluster}.

Define $\mixedBall{\kappa}{\varepsilon} \subset \CC^n$ to be the ball of radius $\varepsilon$ in the mixed norm, i.e., 
$$\|x\|_{(\kappa)} = \sqrt{\sum_{i=1}^{\kappa}|x_i|^2+\sum_{i=\kappa+1}^{n}|x_i|}.$$
The inflation map $S_\kappa$ restricted to $\Ball{\varepsilon}$ is a cover of $\mixedBall{\kappa}{\varepsilon}$ with $2^{n-\kappa}$ sheets and ramification locus $\{x_1x_2\dots x_\kappa=0\}$.  This construction leads to the following:

\begin{corollary}\label{corollary:mixed-ball}
Let $f$ be a square system, $y\in\CC^n$, and 
$V\subset \CC^n$ be a linear subspace of dimension $\kappa$.  Suppose that $g$ is constructed as in \Cref{sec:nearby-singular-system}.  Let $A$ be an affine map such that $g\circ A(0)=0$ and $\ker D(g\circ A(0))=\Span{e_1,\dots,e_\kappa}$.  Let $Q=Q(g\circ A\circ S_\kappa)$ be the quadratic part of $g\circ A\circ S_\kappa$.  If $f\circ A\circ S_\kappa$ and $Q$ satisfy the assumptions of \Cref{lemma:cluster}, then $f$ has $2^\kappa$ zeros in $A(\mixedBall{\kappa}{\varepsilon})$.
\end{corollary}

\subsection{Example} 
Let $\alpha=0.01$ in the system $f(x)=[x_1^2-\alpha^2,x_1+x_2-\alpha x_1^3]^T$.  Let $y:=(0.001,-0.001)$ be a rough approximation to the center of the cluster of zeros of $f$.  Our computation allows us to isolate the cluster of two zeros as depicted in \Cref{fig:example}.

\begin{figure}[ht!]
\floatbox[{\capbeside\thisfloatsetup{capbesideposition={right,center},capbesidewidth=6.5cm}}]{figure}[\FBwidth]
{\caption{
\newline
The zeros of $f_{0.01}=0$ are the red points.  An approximation $y = (0.001, -0.001)$ to the center of the cluster is the blue point.  The shaded region is $A(\mixedBall{\kappa}{0.17})$, which isolates the zeros of $f_{0.01}$.\label{fig:example}}}
{\includegraphics[width=7cm]{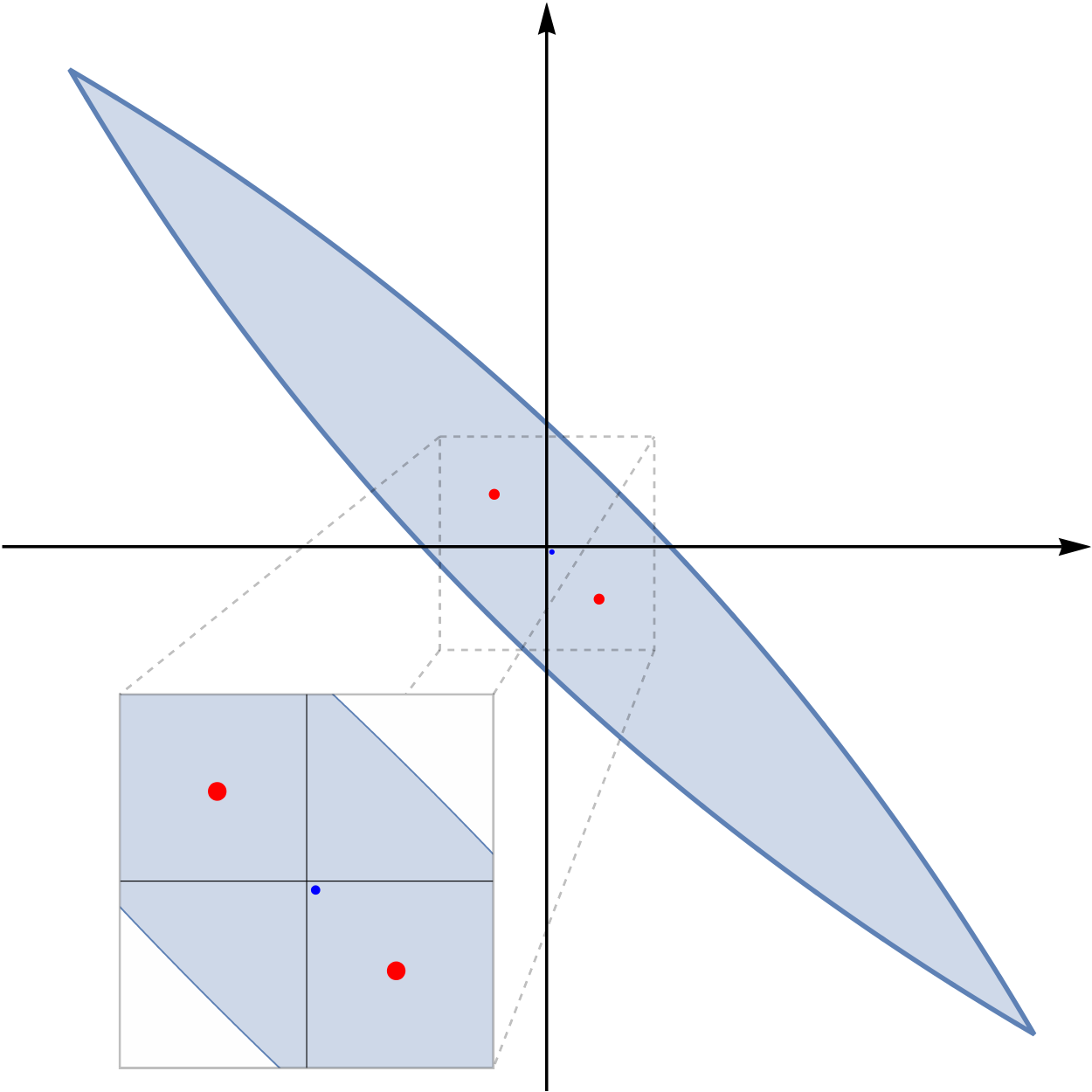}}
\end{figure}

The singular values at $y$ are $1.414$ and $0.001414$.  Moreover, the singular vector corresponding to the smaller singular value is $v:=[-0.707,0.707]^T$.  We then construct $g$ as
\begin{align*}
g(x)&=f(x)-f(y)-D\!f(y)\pi_v(x,y)\\
&\approx[x_1^2-10^{-3}x_1-10^{-3}x_2+10^{-6},-0.01x_1^3+x_1+x_2-2\cdot10^{-9}]^T.
\end{align*}
Using the affine transformation $A$, whose constant part is $y$ and whose linear part is both orthogonal and sends $v$ to $e_1$, we have (after suppressing terms which are nonzero only due to round-off error),
$$
g\circ A\approx\begin{bmatrix}0.5x_1^2+x_1x_2+0.5x_2^2-0.001414x_2\\[.3cm]
\begin{array}{c}
0.003536x_2^3+0.01061x_1^2x_2
+0.01061x_1x_2^2+0.003536x_1^3\\
-1.5\cdot 10^{-5}x_1^2
-3\cdot 10^{-5}x_1x_2
-1.5\cdot 10^{-5}x_2^2
-1.414x_2
\end{array}\end{bmatrix}.
$$
We observe that in this system, there are no constant terms and only $x_2$ appears as a linear term.  We apply the inflation map $S_1$, which replaces $x_2$ by its square, and extract the quadratic part 
$$
Q=Q(g\circ A\circ S_1)\approx [0.5x_1^2-0.001414x_2^2,-1.5\cdot 10^{-5}x_1^2-1.414x_2^2]^T
$$
Using a sum of squares computation, see \Cref{sec:SOS} for details, we find that a lower bound on $q=\|Q\|^2$ on the unit sphere is $0.2221$.  Hence, in \Cref{lemma:cluster}, $c\approx 0.4713$.

We now compute $\varepsilon$ for Condition~\ref{condition1} of \Cref{lemma:cluster} as follows (omitting terms with coefficients less than $10^{-16}$):
$$
f\circ A\circ S_1-Q\approx\begin{bmatrix}0.5x_2^4+x_1x_2^2-0.001414x_1-10^{-4}\\[.3cm]
\begin{array}{c}
0.003536x_2^6+0.01061x_1x_2^4+0.01061x_1^2x_2^2-1.5\cdot 10^{-5}x_2^4\\
+0.003536x_1^3-3\cdot 10^{-5}x_1x_2^2+1.414\cdot 10^{-6}x_1-10^{-11}\\
\end{array}\end{bmatrix}.
$$
Via crude upper bounds, namely, by using the triangle inequality and bounding the magnitude of all variables from above by $\varepsilon$, we find that Condition~\ref{condition1} of \Cref{lemma:cluster} holds for $\varepsilon\in[0.017,0.39]$.  Therefore, by \Cref{corollary:mixed-ball}, $f$ has $2$ zeros in $A(\mixedBall{1}{\varepsilon})$.

\section{Lower bounds on quadratic forms}\label{sec:lower-bound}

A key step in the application of our method is finding a $c>0$ so that Condition \ref{condition2} of \Cref{lemma:cluster} holds.  Computing $c$ is a polynomial optimization problem.  

Let $Q=(Q_1,\dots,Q_n)$ be a vector of quadratic forms.  Condition \ref{condition2} of \Cref{lemma:cluster} requires a positive lower bound on $\|Q\|$ on the unit sphere.  We use a sum-of-squares certificate of positivity approach to compute this bound, but also discuss alternatives. A full study of theoretical efficiency as well as practicality of these approaches is left to future work.  

\subsection{Sum-of-squares relaxation}\label{sec:SOS}

We identify $\CC^n$ with the real space $\RR^{2n}$ of twice as many dimensions.  For each variable $x_i$, we let $a_i=\Re(x_i)$ and $b_i=\Im(x_i)$.  We are interested in a positive lower bound on the following quartic over the real unit sphere $\Sphere_1 \subset \RR^{2n} $:
\[
q(a,b) = \|Q(a+\ii b)\|^2 = \sum_{j=1}^n\Re(Q_j(a+\ii b))^2 + \Im(Q_j(a+\ii b))^2.
\]
In addition, we define 
$$
s(a,b)=\|a+\ii b\|^2=\sum a_i^2+\sum b_i^2,
$$
where the real zero set of $s(a,b)-1$ is the unit sphere.  Our goal is to compute 
$$q^* = \min_{\Sphere_1} q = \min \{q(a,b)\mid s(a,b)=1\}.$$

Let $\Sigma_{SOS}$ be the convex cone of polynomials that are sums of squares. The \DEF{sums-of-squares moment relaxation of level $r$} computes a lower bound on $q^\ast$, see, e.g., \cite[\S 3.2.4]{blekherman2012semidefinite} and \cite{KMoment:Schmudgen}. In our case, this bound is  
$$
c_r = \max \{ c \mid q-c = \sigma + h ,\ \sigma\in\Sigma_{SOS},\ h\in\ideal{s-1},\ \max(\deg \sigma,\deg h) \leq 2r \}.
$$

Since $\lim_{r\to\infty} c_r=q^*$, for $r$ large enough, $c_r>0$ when $q^*>0$.
We remark that, should a relaxation of level $r$ succeed with $c_r>0$, a rational certificate of positivity for $q$ can be produced as in~\cite{peyrl2008computing}.

\subsection{Alternative approaches}

Isolating all critical points of $q$ on the sphere $\Sphere_1$ is one alternate approach. This can be accomplished provided the zeros of a Lagrange multiplier system are regular and their number matches the generic number of zeros of this system. In full generality, the zeros may be singular, ``at infinity'', or poorly conditioned, which would hamper the practicality of this approach.   

Interval arithmetic, see \cite{Moore:Interval}, may be used to estimate a lower bound on $q$ over the sphere $\Sphere_1$.  This method estimates $q$ on each region of a partition of the sphere.  As the partition becomes finer, this estimate converges to $q^\ast$.  Subdivision-based schemes involving interval arithmetic sometimes have exponential complexities, see \cite{Burr:DiameterDistnace}, but may be efficient in practice \cite{Cucker:PV:average}.

\section{Summary and remarks}

Given a heuristic to determine an approximate rank and an approximate kernel of linear map, we summarize our approach below.

\medskip

\noindent{\bf Input:} A polynomial system $f:\CC^n\to \CC^n$ and a point $y\in \CC^n$ that approximates a cluster of zeros.  \\
\noindent{\bf Output:} An integer $0\leq \kappa\leq n$, a real number $\varepsilon>0$, and the image of a mixed-norm ball $\mixedBall{\kappa}{\varepsilon}$ that contains $2^\kappa$ zeros of $f$.
\begin{enumerate}
\item Compute $\kappa$, the dimension of an approximate kernel of $D\!f(y)$.  
\item Use \Cref{eq:singular} to construct the system $g$ which has a zero at $y$ of Type $\Sigma^\kappa$.
\item Construct an affine transformation $A$ so that $g\circ A$ has a zero at $0$ and the kernel of $g\circ A$ at $0$ is $\langle e_1,\dots,e_\kappa\rangle$.
\item Use the inflation operator $S_\kappa$ so that $g\circ A\circ S_\kappa$ has no constant or linear terms.
\item Find a lower bound $c>0$ for the norm of the quadratic part $Q$ of the system $g\circ A\circ S_1$ on the unit sphere using an approach from \Cref{sec:lower-bound}. 
\item Find a small $\varepsilon>0$ so that $\|f-Q\| < c\varepsilon^2$ on the sphere $\Sphere_{\varepsilon}$.
\item Return $\kappa$, $A$, and $\varepsilon$. 
\end{enumerate}

\subsection{Application to systems with exact singular zeros} \label{remark:exact-zeros}
Suppose $f$ has an exact zero $x$ of singularity Type $\Sigma^\kappa$ inflatable to a regular quadratic zero. 
As long as there is an efficient way to refine the approximation $y$ of $x$, our approach  is guaranteed to succeed.
Indeed, one can use SVD to recover $\kappa$ and and an approximation of $\ker D\!f(x)$ robustly, since exactly $\kappa$ singular values approach $0$ as $y\to x$.

This procedure does not certify the fact that there is a singular zero in the constructed region, but it isolates and certifies the existence of a cluster of multiplicity equal to that of $x$.

\subsection{Approximating the cluster}

While we consider computing an approximation $y$ to a cluster as a separate problem, one algorithm we have in mind is that of deflation, as described in~\cite{LVZ}. This approach provides a way to refine an approximation to an exact singular zero as required in~\Cref{remark:exact-zeros}.   
In the case of a cluster of several distinct zeros, deflation still delivers an augmented (often overdetermined) system of equations that can be used for our purpose. The fixed point of the operator in Newton's method is a good heuristic approximation of the cluster even if this fixed point is not an actual zero of the augmented system. 

\subsection{Geometry of a cluster}
Inflation changes the geometry around a cluster. As a result, the mixed-norm ball of \Cref{corollary:mixed-ball} approximates the geometry of a cluster better than a Euclidean ball, as in the cluster isolation statement in ~\cite[Theorem~4]{dedieu2001simple}.  It may be possible to pursue a strategy similar to that of~\cite{dedieu2001simple} in the case of a regular quadratic zero. We remark that even if this strategy succeeds, approaches inspired by Smale's $\alpha$-theory tend to provide very conservative bounds.  

\subsection{Reduction to regular zeros of higher order}
Our isolation technique can be readily adopted to the case of a regular zero of arbitrary order. In fact, our method has a straightforward application to simple multiple zeroes of arbitrary multiplicity as studied, for instance, in~\cite{HaoJiangLiZhi:2020simple-multiple-zeros}. In this case, $\kappa=1$ and one may generalize inflation to produce a system with a regular zero of order $\mu$ when the multiplicity $\mu$ is known. This construction would also result in a mixed norm ball that may give a tighter isolation region than the Euclidean ball produced by the method of~\cite{HaoJiangLiZhi:2020simple-multiple-zeros}. 

A system with a singular zero with arbitrary multiplicity structure could be more complicated, however.  
We suspect that a more intricate inflation procedure may be needed in the general case. 

\subsection{Applications to square polynomial systems}

Given a system of $n$ polynomial equations in $n$ unknowns, it is common to have both a heuristic method to approximate all zeros and a way to count them with multiplicity. When, in addition, all zeros are either regular or inflatable to regular quadratic zeros, our approach certifies the completeness of the approximate computation by producing a computer derived, but explicitly checkable proof.

\bibliographystyle{plain}
\bibliography{bibliography.bib}

\end{document}